%% file: apostmax.tex
\documentclass[12pt,a4paper,twoside]{article}

\title{\sc Two-Sided A Posteriori Error Bounds for Electro-Magneto Static Problems}
\def\shorttitle{Two-Sided A Posteriori Error Bounds for Electro-Magneto Static Problems}
\def\pauthor{Dirk Pauly and Sergey Repin}

\def\mylabelonoff{off}
\def\allowdisbrk{no}

\input head_paper

\input head_dirk_sergey

\begin{document}

\maketitle{}

\begin{center}
{\sf Dedicated to the anniversary of Prof. Nina Nikolaevna Uraltseva}
\end{center}

\begin{abstract}
This paper is concerned with the derivation of computable and guaranteed
upper and lower bounds of the difference between the exact and the approximate solution
of a boundary value problem for static Maxwell equations.
Our analysis is based upon purely functional argumentation
and does not attract specific properties of an approximation method.
Therefore, the estimates derived in the paper at hand are applicable
to any approximate solution that belongs to the corresponding energy space.
Such estimates (also called error majorants of the functional type)
have been derived earlier for elliptic problems \cite{repinconvex,RdeGruyter}.\\
\\
\amsclass{65 N 15, 78 A 30}
\end{abstract}

\tableofcontents

\section{Introduction and notation}

The main goal of this paper is to derive guaranteed and computable
upper and lower bounds of the difference between the exact solution
of an electro-magneto static boundary value problem and any approximation
from the corresponding energy space. We discuss the method with the paradigm
of a prototypical electro-magneto static problem in a bounded domain.
The generalized formulation is given by the integral identity \eqref{varformu}.
We show that (as in many other problems of mathematical physics) certain transformations
of \eqref{varformu} lead to guaranteed and fully computable majorants and
minorants of the approximation error. However, the case considered here
has special features that make (at some points) the derivation procedure
different from, e.g., that which has been earlier applied to other elliptic type problems.
This happens because the corresponding differential operator has a nonzero
kernel (which contains curl-free vector fields) and the set
of trial functions in \eqref{varformu} is restricted to a rather special
subspace. For these reasons, the derivation of the estimates is based on
Helmholtz-Weyl decompositions of vector fields, 
orthogonal projections onto subspaces, 
and on a certain version of a Poincar\'e-Friedrich
estimate for the differential operator $\curl$.
First, we show that the distance between the exact solution $E$
and the approximate solution $\tilde{E}$ (measured throughout the semi-norm
generated by the operator $\curl$) is equal to some norm of the so-called
{\it residual functional} $\ell_{\tilde{E}}$ (cf. \eqref{stepone}). 
If $\tilde{E}$ satisfies the boundary condition exactly,
i.e. $\ttrg\tilde{E}=G$, then the latter functional vanishes 
if and only if $\curl\tilde{E}$ coincides with $\curl E$.
Lemma \ref{estimatelem} shows that an error majorant can be expressed 
throughout a certain norm of $\ell_{\tilde{E}}$ (cf. \eqref{estimatetheoestone}). 
However, in general, computing
of this norm is hardly possible because it requires finding a supremum 
over an infinite number of vector fields.

Theorem \ref{esttwo} provides a computable form of the upper bound. 
The corresponding estimate \eqref{estprotwo} 
shows that the error majorant is the sum of five
terms, which can be thought of as penalties for possible violations
of the relations \eqref{curlcurleq}-\eqref{orthdi}.
It contains only known vector fields and global constants 
depending on geometrical properties of the domain.
Moreover, it is easy to see that the upper bound vanishes
if and only if $\tilde{E}$ coincides
with the exact solution $E$ and a 'free variable' $Y$
encompassed in the estimate coincides with $\mu^{-1}\curl\tilde{E}$.
Also, we show that the estimates derived are sharp in the sense that the
estimates \eqref{estprothree} and \eqref{estproonem}
have no irremovable gap between the left and
right hand sides (Remark \ref{estthreeremthree}). 
Finally, in Section \ref{lowboundsec}, we derive lower estimates of the
difference between exact and approximate solutions. The corresponding result
is presented by Theorem \ref{lowerboundtheo}. 
This estimate is also computable, guaranteed and
sharp provided that the approximation exactly satisfies the prescribed boundary condition.

Throughout this paper, we consider a bounded domain
$\Omega\subset\rt$ with Lipschitz continuous boundary $\gamma$ 
and denote the corresponding outward unit normal vector by $n$. 
$E$ and $H$ stand for the
electric and magnetic vector fields, respectively, while $\eps$
and $\mu$ denote positive definite, symmetric matrices with
measurable, bounded coefficients that describe properties of the
media (dielectricity and permeability, respectively). 
For the sake of brevity, matrices (matrix-valued functions) 
with such properties are called 'admissible'. 
We note that the corresponding inverse matrices are
admissible as well. In particular, there exists a constant $c_{\mu}>0$, 
such that for a.e. $x\in\Omega$
\begin{align}
\mylabel{cmu}
c_{\mu}|\xi|^2\leq\mu^{-1}(x)\xi\cdot\xi,\quad \forall\,\xi\in\rt.
\end{align}
By $\Ltom$ we denote the usual scalar $\Lt$-Hilbert space
of square integrable functions over $\Omega$
and by $\Hom{}$ the Hilbert space of real-valued $\Lt$-vector fields, i.e. $\Lt(\Omega,\rt)$.
For the sake of simplicity we restrict our analysis to the case of real valued
functions and vector fields. The generalization to complex valued spaces
is straight forward.

Orthogonality and the
orthogonal sum with respect to the scalar product of $\Hom{}$ is
denoted by $\bot$ and $\oplus$, respectively, i.e. $\Phi\bot\Psi$ if
$$\scpom{\Phi}{\Psi}:=\int_{\Omega}\Phi\cdot\Psi\,d\lambda=0,$$
where $\lambda$ denotes Lebesgue's measure. 
Moreover, by $\bot_\nu$ (respectively
$\oplus_\nu$) we indicate the orthogonality (respectively orthogonal
sum) in terms of the weighted $\Lt$-scalar product
$\scpom{\nu\Phi}{\Psi}$ generated by an admissible matrix $\nu$.

Throughout the paper we will utilize the following
functional spaces:
\begin{align*}
\Hcom{}&:=\Set{\Psi\in\Hom{}}{\curl\Psi\in\Hom{}},\\
\Hczom{}&:=\Set{\Psi\in\Hcom}{\curl\Psi=0},\\
\Hccom{}&:=\ol{\Cizom},\quad\text{closure in }\Hcom{},\\
\Hcczom{}&:=\Hccom{}\cap\Hczom{}.
\end{align*}
Analogously, we define the spaces associated with the operators $\div$ and $\grad$.
Furthermore, we introduce the spaces (containing the so-called Dirichlet and Neumann fields)
\begin{align*}
\harmdieps&:=\Hcczom{}\cap\eps^{-1}\Hdzom{}\\
&\,\,=\Set{\Psi\in\Hom{}}{\curl\Psi=0,\,\div\eps\Psi=0,\,n\times\Psi|_{\gamma}=0},\\
\harmnemu&:=\Hczom{}\cap\mu^{-1}\Hdczom{}\\
&\,\,=\Set{\Psi\in\Hom{}}{\curl\Psi=0,\,\div\mu\Psi=0,\,n\cdot\mu\Psi|_{\gamma}=0}.
\end{align*}
Here and later on we write
$E\in\eps^{-1}\Hdzom{}$ if $\eps E\in\Hdzom{}$.
These are finite dimensional spaces, whose dimensions are denoted
by $d_D$ and $d_N$, respectively.
In fact, these numbers are equal to the so-called Betti numbers
of $\Omega$ and depend only on topological properties of the domain
(for a detailed presentation see \cite{paulystatic}).
A basis of $\harmdieps$ shall be given by special vector fields $\{H_{1},\dots,H_{d_D}\}$.

Finally, we note that being equipped with the proper inner products 
all the above introduced functional spaces are Hilbert spaces.

The classical formulation of the electro-magneto static problem
for a given vector field $F$ (driving force) and given  $\eps$, $\mu$ reads as follows:
Find a magnetic field
$$H\in\Hcom{}\cap\mu^{-1}\Hdczom{}\cap\harmnemu^{\bot_\mu}$$
and a corresponding electric field
$$E\in\Hccdzbotom{}:=\Hccom{}\cap\eps^{-1}\Hdzom{}\cap\harmdieps^{\bot_\eps},$$
such that in $\Omega$
$$\curl H=F,\quad\curl E=\mu H.$$
In other words, the problem is to find vector fields
$H\in\Hcom{}\cap\mu^{-1}\Hdom{}$
and 
$$E\in\Hcdepsom{}:=\Hcom{}\cap\eps^{-1}\Hdom{},$$ 
such that
\begin{align*}
\curl H&=F,&\curl E&=\mu H&&\text{in }\Omega,\\
\div\mu H&=0,&\div\eps E&=0&&\text{in }\Omega,\\
n\cdot\mu H|_\gamma&=0,&n\times E|_\gamma&=0&&\text{on }\gamma,\\
\mu H&\,\,\bot\,\,\harmnemu,&\eps E&\,\,\bot\,\,\harmdieps,
\end{align*}
where the homogeneous boundary conditions are to be understood in the weak sense.

This coupled problem is equivalent to an electro-magneto static Maxwell
problem in second order form,
which in classical terms reads as follows:
Find an electric field
$E$ in $\Hccdzbotom{}$,
such that $\mu^{-1}\curl E$ belongs to $\Hcom{}$ and
$$\curl\mu^{-1}\curl E=F$$
holds in $\Omega$, i.e. find
$E\in\Hcdepsom{}$,
such that $\mu^{-1}\curl E\in\Hcom{}$ and
\begin{align}
\curl\mu^{-1}\curl E&=F&&\text{in }\Omega,\mylabel{hombcprobeqone}\\
\div\eps E&=0&&\text{in }\Omega,\mylabel{hombcprobeqtwo}\\
n\times E|_\gamma&=0&&\text{on }\gamma,\mylabel{hombcprobeqthree}\\
\eps E&\,\,\bot\,\,\harmdieps.\mylabel{hombcprobeqfour}
\end{align}
Once $E$ has been found, the magnetic field is given by $H:=\mu^{-1}\curl E$.

We note that the problem
\begin{align*}
\curl\mu^{-1}\curl E+\kappa^2 E&=F&&\text{in }\Omega,\\
n\times E|_\gamma&=0&&\text{on }\gamma
\end{align*}
with positive $\kappa$ was considered in \cite{Anjam} 
in the context of functional type a posteriori error estimates. 
From the mathematical point of view, this problem
is much simpler as the problem \eqref{hombcprobeqone}-\eqref{hombcprobeqfour}
since the zero order term makes the underlying operator positive definite.

\section{Variational formulation and solution theory}

Henceforth, we consider \eqref{hombcprobeqone}-\eqref{hombcprobeqfour} 
assuming that the boundary condition on $\gamma$ may be inhomogeneous 
(physically, such a condition is motivated by the presence of electric currents on the boundary).
Hence, we intend to discuss the following prototypical
electro-magneto static Maxwell problem in second order form:
Find an electric field 
$E$ in $\Hcdepsom{}$,
such that
\begin{align}
\curl\mu^{-1}\curl E&=F&&\text{in }\Omega,\mylabel{curlcurleq}\\
\div\eps E&=0&&\text{in }\Omega,\mylabel{diveq}\\
n\times E|_\gamma&=G&&\text{on }\gamma,\mylabel{bdcond}\\
\eps E&\,\,\bot\,\harmdieps,\mylabel{orthdi}
\end{align}
i.e., find $E$ in 
$$\Hcdzbotom{}:=\Hcom{}\cap\eps^{-1}\Hdzom{}\cap\harmdieps^{\bot_\eps}$$
satisfying \eqref{curlcurleq} and \eqref{bdcond}.
There are at least two methods to prove existence of the solution.
One is based upon Helmholtz-Weyl decompositions
(see, e.g.
\cite{milani,potential,boundaryelectro,decomposition,paulystatic,paulydeco}).
The second method consists of introducing and studying a suitable generalized
statement  of the problem \eqref{curlcurleq}-\eqref{orthdi}.
In this paper, we use the second method 
because it provides a natural way of deriving error estimates.
Both methods are based on Poincar\'e-Friedrich estimates,
see Remark \ref{poincaregenest},
and (if it is needed) exploit suitable extension operators
for the boundary data.
On this way, we also need a certain
version of the Poincar\'e-Friedrich estimate, namely
\begin{align}
\normom{\Psi}\leq c_{p}\normom{\curl\Psi}\quad\forall\,\Psi\in\Hccdzbotom{}.\mylabel{Poincare}
\end{align}
Of course, there exist more general variants of Poincar\'e-Friedrich's estimate 
\eqref{Poincare} for vector fields. Here, we refer to Remark \ref{poincaregenest}. 

Now, let $\Eg$ be some vector field in $\Hcdzbotom{}$
satisfying the boundary condition \eqref{bdcond} in the generalized sense,
i.e., $E-\Eg\in\Hccom{}$.
The generalized solution
$$E\in\Hccdzbotom{}+\Eg\subset\Hcdzbotom{}$$
of \eqref{curlcurleq}-\eqref{orthdi}
is then defined by the relation
\begin{align}
\scpom{\mu^{-1}\curl E}{\curl W}&=\scpom{F}{W}\quad\forall\,W\in\Hccdzbotom{}.\mylabel{varformu}
\end{align}
If $F\in\Hom{}$ then by the Cauchy-Scharz inequality the right hand
side of \eqref{varformu} is a linear and continuous functional over $\Hccdzbotom{}$. 
By \eqref{Poincare} the left hand side of
\eqref{varformu} is a strongly coercive bilinear form over $\Hccdzbotom{}$. 
Thus, under these assumptions the problem
\eqref{varformu} is uniquely solvable in $\Hccdzbotom{}+\Eg$ by Lax-Milgram's theorem.

First, we note some Helmholtz-Weyl decompositions of $\Hom{}$,
i.e. decompositions into solenoidal and curl-free fields,
which will be used frequently throughout our analysis.

\begin{lem}
\mylabel{helmholtzlem}
$\Hom{}$ can be decomposed as
\begin{align}
\Hom{}
&=\eps\Hcczom{}\oplus_{\eps^{-1}}\ol{\curl\Hcom{}}\nonumber\\
&=\eps\ol{\grad\Hgcom{}}\oplus_{\eps^{-1}}\Hdzom{}\nonumber\\
&=\eps\ol{\grad\Hgcom{}}\oplus_{\eps^{-1}}\eps\harmdieps\oplus_{\eps^{-1}}\ol{\curl\Hcom{}}\mylabel{helmdecoepsmo}
\intertext{and}
\Hom{}
&=\Hcczom{}\oplus_{\eps}\eps^{-1}\ol{\curl\Hcom{}}\nonumber\\
&=\ol{\grad\Hgcom{}}\oplus_{\eps}\eps^{-1}\Hdzom{}\nonumber\\
&=\ol{\grad\Hgcom{}}\oplus_{\eps}\harmdieps\oplus_{\eps}\eps^{-1}\ol{\curl\Hcom{}},
\mylabel{helmdecoeps}
\end{align}
where all closures are taken in $\Hom{}$ and $\Hgcom{}=\Hozom$.
Moreover, 
$$\ol{\curl\Hcom{}}=\Hdzbotom{}:=\Hdzom{}\cap\harmdieps^{\bot}.$$
\end{lem}

\begin{rem}
\mylabel{helmholtzrem}
Let us denote the $\eps$-orthogonal projection onto $\eps^{-1}\ol{\curl\Hcom{}}$ 
in \eqref{helmdecoeps} by $\pi$. Then we have for all $\Phi\in\Hcom{}$
\begin{align}
\ttrg\pi\Phi=\ttrg\Phi,\quad\curl\pi\Phi=\curl\Phi\mylabel{pipropo}
\end{align}
and for all $\Psi\in\Hom{}$
\begin{align*}
\div\eps\pi\Psi=0,\quad\eps\pi\Psi\bot\harmdieps,\quad
\curl(1-\pi)\Psi=0,\quad\ttrg(1-\pi)\Psi=0.
\end{align*}
The latter line can be written in a more compact and precise way as
$$\pi\Hom{}=\eps^{-1}\ol{\curl\Hcom{}}=\Hdzbotepsom{},\quad(1-\pi)\Hom{}=\Hcczom{}.$$
\end{rem}

\begin{rem}
\mylabel{projectionrem}
Note that by \eqref{curlcurleq} $F$ must be solenoidal and
perpendicular in $\Hom{}$ to $\Hcczom{}$. 
Using the Helmholtz-Weyl decomposition \eqref{helmdecoepsmo}
we decompose the vector field $F\in\Hom{}$, 
i.e. $F=\eps F_{D}+\eps F_{\grad}+F_{\curl}$. 
Then, for any $W\in\Hccdzbotom{}$ we compute
$\scpom{F}{W}=\scpom{F_{\curl}}{W}$. Hence, the functional on the
right hand side of \eqref{varformu} can not distinguish between $F$
and the projection $F_{\curl}$. 
\end{rem}

The following theorem states the main existence result.

\begin{theo}
\mylabel{solutiontheo}
Let $F\in\Hdzbotom{}$ and let $\Eg\in\Hcdzbotom{}$ satisfy the boundary condition \eqref{bdcond}.
Then the boundary value problem \eqref{curlcurleq}-\eqref{orthdi}
is uniquely weakly solvable in $\Hccdzbotom{}+\Eg$. The solution operator is continuous.
\end{theo}

\begin{rem}
\mylabel{solutiontheorem}
The kernel of \eqref{curlcurleq}-\eqref{bdcond} equals $\harmdieps$.
We only have to show $\curl E=0$ but this follows
immediately since $E\in\Hccom{}$ and thus,
$$0=\scpom{\curl\mu^{-1}\curl E}{E}=\scpom{\mu^{-1}\curl E}{\curl E}.$$
\end{rem}

\begin{rem}
\mylabel{traceextension}
The boundary data $G$ and its extension $\Eg$ can be described in more detail.
Since the papers \cite{alonsovalli,buffaciarlet,buffacostabelsheen}
and the more general paper of Weck \cite{wecklip} we know that
even for Lipschitz domains, where the non scalar trace business is a challenging task,
there exist a bounded linear tangential trace operator $\ttrg$
and a corresponding bounded linear tangential extension operator $\textg$ (right inverse)
mapping $\Hcom{}$ to special tangential vector fields on the boundary, i.e.
$$\Htmotcg:=\Set{\psi\in\Hgen{-1/2}{t}{}(\gamma)}{\curls\psi\in\Hgen{-1/2}{}{}(\gamma)},$$
and vice verse.
Here, $\curls$ denotes the surface $\curl$.
Using the Helmholtz-Weyl decomposition \eqref{helmdecoeps}
we even get an improved extension operator. We have
\begin{align*}
\ttrg:\Hcom{}&\to\Htmotcg,\\
\textg:\Htmotcg&\to\Hcdzbotom{}.
\end{align*}
Applied to smooth vector fields we have $\ttrg=n\times\,\cdot\,|_{\gamma}$.
Now, we may specify the boundary data $G\in\Htmotcg$
and the extension $\Eg:=\textg G\in\Hcdzbotom{}$
as well as our variational formulation for $E=E^\circ+\textg G$:
Find $E^\circ\in\Hccdzbotom{}$, such that
$$b(E^\circ,W):=\scpom{\mu^{-1}\curl E^\circ}{\curl W}
=\scpom{F}{W}-\scpom{\mu^{-1}\curl\textg G}{\curl W}=:\ell(W)$$
holds for all $W\in\Hccdzbotom{}$.
\end{rem}

\begin{rem}
\mylabel{traceextensiontwo}
Henceforth, we assume that $G$ is given by a tangential trace of
some vector field $\check{G}\in\Hcom{}$.
\end{rem}

\begin{rem}
\mylabel{poincaregenest}
More general variants of the Poincar\'e-Friedrich estimate for vector fields
\eqref{Poincare} are known. For instance, we have
\begin{align*}
\normom{\Psi}
\leq c_{p}\left(\normom{\curl\Psi}
+\normom{\div\eps\Psi}
+\norm{\ttrg\Psi}_{\Htmotcg}
+\sum_{n=1}^{d_{D}}\left|\scpom{\eps\Psi}{H_{n}}\right|\right),
\end{align*}
which holds for all $\Psi\in\Hcdepsom{}$.
This estimate may be proved by an indirect argument using
a 'Maxwell compact embedding property' of $\Omega$,
which holds true not only for Lipschitz domains,
but also, if the homogeneous boundary condition is considered,
for more irregular domains (cone properties), see \cite{xmas}.
For inhomogeneous boundary conditions
the Lipschitz assumption can not be weakened.
Actually, it is just the continuity of the solution operator
of the corresponding electro static boundary value problem,
see \cite{peter,kuhnpaulytracerep,kuhnpaulyreg}.
\end{rem}

\section{Upper bounds for the deviation from the exact solution}

Let $\tilde{E}$ be an approximation of
$E\in\Hccdzbotom{}+\Eg\subset\Hcdzbotom{}$. 
We assume that $\tilde{E}$ belongs
to $\Hcdepsom{}$, which means that, in general, the boundary
condition, the divergence-free condition, and the orthogonality to
the Dirichlet fields might be violated, i.e. the approximation field
may be such that
$$\ttrg\tilde{E}\neq G,\quad\div\eps\tilde{E}\neq0,
\quad\scp{\eps\tilde{E}}{H}\neq0\quad\text{for some }H\in\harmdieps.$$
Moreover, for the subsequent analysis 
and then also for the numerical application, which is even more important, 
it is sufficient to assume just $\tilde{E}\in\Hcom{}$.

Our goal is to obtain upper bounds for the difference
between $\curl E$ and $\curl\tilde{E}$ in terms of the weighted norm
$$\normmumoom{\Psi}:=\normom{\mu^{-1/2}\Psi}=\scpom{\mu^{-1}\Psi}{\Psi}^{1/2}.$$
First, we use \eqref{varformu} and get for all $W\in\Hccdzbotom{}$
\begin{align}
\scpom{\mu^{-1}\curl(E-\tilde{E})}{\curl W}
&=\scpom{F}{W}-\scpom{\mu^{-1}\curl\tilde{E}}{\curl W}=:\ell_{\tilde{E}}(W),\mylabel{stepone}
\end{align}
where $\ell_{\tilde{E}}$ is a linear and continuous functional over $\Hccdzbotom{}$
as well as over $\Hcom{}$.
Furthermore, $\ell_{\tilde{E}}$ does not depend on the exact solution $E$. 

\begin{rem}
\mylabel{remell}
Obviously, $\ell_{\tilde{E}}$ vanishes if $\curl E=\curl\tilde{E}$.
Furthermore, if $\tilde{E}$ satisfies the boundary condition exactly,
i.e. $\ttrg\tilde{E}=G$, then $\ell_{\tilde{E}}=0$ if and only if
$\curl E=\curl\tilde{E}$ (or what is equivalent if and only if $E=\pi\tilde{E}$). 
This holds by the following argument using the Helmholtz-Weyl decomposition:
If $\ttrg\tilde{E}=G$ then $E-\pi\tilde{E}\in\Hccdzbotom{}$. 
Thus $\curl(E-\pi\tilde{E})=0$ by $\ell_{\tilde{E}}=0$. 
But then $E-\pi\tilde{E}$ is a Dirichlet field and hence must vanish by orthogonality. 
Finally $\curl\pi\tilde{E}=\curl\tilde{E}$.
\end{rem}

The second step is based upon the following result:

\begin{lem}
\mylabel{estimatelem}
Let $E\in\Hcdzbotom{}$ be the exact solution and $\tilde{E}\in\Hcom{}$
be an approximation. Furthermore, let $\ell_{\tilde{E}}$ be as above
and let $c_\ell>0$ exist, such that
$$\scpom{\mu^{-1}\curl(E-\tilde{E})}{\curl W}
=\ell_{\tilde{E}}(W)\leq c_\ell\normmumoom{\curl W}$$
holds for all $W\in\Hccdzbotom{}$. Then
\begin{align}
\normmumoom{\curl(E-\tilde{E})}&\leq c_\ell+2\normmumoom{\curl T}\mylabel{estimatetheoestone}
\intertext{holds for all $T\in\Hcom{}$, for which the tangential trace
coincides with the tangential trace of $E-\tilde{E}$, i.e. $G-\ttrg\tilde{E}$,
on the boundary $\gamma$.
If additionally $\ttrg\tilde{E}=G$ then}
\normmumoom{\curl(E-\tilde{E})}&\leq c_\ell.\mylabel{estimatetheoesttwo}
\end{align}
\end{lem}

\begin{proof}
We use the Helmholtz-Weyl decomposition \eqref{helmdecoeps}
and the projection $\pi$ from Remark \ref{helmholtzrem}.
We consider a vector field $T\in\Hcom{}$ with
$\ttrg T=G-\ttrg\tilde{E}$ and define the vector field
$$W:=E-\pi(T+\tilde{E})=E-\tilde{E}+
(1-\pi)\tilde{E}-\pi T\in\Hccdzbotom{},$$
which holds by \eqref{pipropo}.
Hence, $\curl W=\curl(E-\tilde{E})-\curl T$.
Using Cauchy-Schwarz' inequality we obtain
\begin{align*}
\normmumoom{\curl W}^2
&=\scpom{\mu^{-1}\curl(E-\tilde{E})}{\curl W}-\scpom{\mu^{-1}\curl T}{\curl W}\\
&\leq\left(c_\ell+\normmumoom{\curl T}\right)\normmumoom{\curl W}
\end{align*}
and thus $\normmumoom{\curl W}\leq c_\ell+\normmumoom{\curl T}$.
By the triangle inequality we get \eqref{estimatetheoestone}.
\eqref{estimatetheoesttwo} is trivial setting $T:=0$.
\end{proof}

Using the trace and extension operators from
Remark \ref{traceextension} we obtain the following result:

\begin{cor}
\mylabel{estimatecor}
Let the assumptions of Lemma \ref{estimatelem} be satisfied. Then
\begin{align}
\begin{split}
\normmumoom{\curl(E-\tilde{E})}&\leq c_\ell+2\normmumoom{\curl\textg(G-\ttrg\tilde{E})}\\
&\leq c_\ell+2c_\gamma\norm{G-\ttrg\tilde{E}}_{\Htmotcg}.
\end{split}\mylabel{estcurltrace}
\end{align}
Here $c_\gamma>0$ is the constant in the inequality
\begin{align}
\normmumoom{\curl\textg\psi}\leq c_\gamma\norm{\psi}_{\Htmotcg}
\quad\forall\psi\in\Htmotcg.\mylabel{extensionineq}
\end{align}
\end{cor}

\begin{proof}
Setting $T:=\textg(G-\ttrg\tilde{E})$ in \eqref{estimatetheoestone} and using
\eqref{extensionineq} proves \eqref{estcurltrace}.
We note that \eqref{estimatetheoesttwo} follows directly from the corollary as well.
\end{proof}

Lemma \ref{estimatelem} and Corollary \ref{estimatecor} imply the following result.

\begin{theo}
\mylabel{estone}
Let $E,\tilde{E}$ be as in Lemma \ref{estimatelem}. Then
\begin{align}
\begin{split}
&\qquad\normmumoom{\curl(E-\tilde{E})}\\
&\leq\frac{c_p}{\sqrt{c_{\mu}}}\normom{F-\curl Y}+\normmuom{Y-\mu^{-1}\curl\tilde{E}}
+2c_\gamma\norm{G-\ttrg\tilde{E}}_{\Htmotcg},
\end{split}\mylabel{estproone}
\end{align}
where $Y$ is an arbitrary vector field in $\Hcom{}$.
\end{theo}

\begin{proof}
For any $Y\in\Hcom{}$ and any $W\in\Hccom{}$ we have
\begin{align}
-\scpom{\curl Y}{W}+\scpom{Y}{\curl W}&=0.\mylabel{steptwo}
\end{align}
Combining \eqref{stepone} and \eqref{steptwo}, we obtain for all $W\in\Hccdzbotom{}$
\begin{align}
\begin{split}
&\qquad\scpom{\mu^{-1}\curl(E-\tilde{E})}{\curl W}\\
&=\scpom{F-\curl Y}{W}+\scpom{Y-\mu^{-1}\curl\tilde{E}}{\curl W}=\ell_{\tilde{E}}(W).
\end{split}\mylabel{stepthree}
\end{align}
By Cauchy-Schwarz' inequality, Poincar\'e-Friedrich's estimate \eqref{Poincare} and \eqref{cmu}
we estimate the right hand side $\ell_{\tilde{E}}(W)$ of \eqref{stepthree}
\begin{align}
\begin{split}
\left|\scpom{F-\curl Y}{W}\right|&\leq\normom{F-\curl Y}\normom{W}
\leq c_p\normom{F-\curl Y}\normom{\curl W}\\
&\leq\frac{c_p}{\sqrt{c_{\mu}}}\normom{F-\curl Y}\normmumoom{\curl W},
\end{split}\mylabel{estimateone}\\
\left|\scpom{Y-\mu^{-1}\curl\tilde{E}}{\curl W}\right|
&\leq\normmuom{Y-\mu^{-1}\curl\tilde{E}}\normmumoom{\curl W}.
\mylabel{estimatetwo}
\end{align}
Now, Lemma \ref{estimatelem} and Corollary \ref{estimatecor}complete the proof.
\end{proof}

We remark that the latter estimate is unable 
to measure adequately the deviation
of the divergence of $\eps\tilde{E}$ to $0$ (this is obvious since
$\eps\tilde{E}$ even does not need to have any divergence).
On the other hand, even if $\div\eps\tilde{E}\not=0$ then
the semi-norm $\normmumoom{\curl\,\cdot\,}$ could not feel
the lack of the constraint $\div\eps\tilde{E}=0$.
The same holds true for the deviation of $\eps\tilde{E}$
from the orthogonality to the Dirichlet fields.
However, it is not difficult to transform the estimate into a form,
in which the estimate is represented in terms of the semi-norm 
\begin{align}
\tnormom{\Psi}:=\|\curl\Psi\|_{\mu^{-1},\Omega} +\normom{\div\eps\Psi}
+\sum_{n=1}^{d_{D}}\left|\scpom{\eps\Psi}{H_{n}}\right|\mylabel{tnormdef}
\end{align}
on $\Hcdepsom{}$, which obviously is a norm on 
$$\Hccdepsom{}:=\Hccom{}\cap\eps^{-1}\Hdom{}.$$

\begin{rem}
\mylabel{normtrem}
These two facts can be seen applying the Helmholtz-Weyl decomposition
\eqref{helmdecoeps} and the projection $\pi$.
In particular, by \eqref{pipropo} replacing $\tilde{E}$ by $\pi\tilde{E}$
in Theorem \ref{estone} would change nothing.
In other words, the part of $\tilde{E}$ containing the eventually non vanishing divergence term
$(1-\pi)\tilde{E}$ can be added to any term in \eqref{estproone} without changing anything.
To get the 'full' norm \eqref{tnormdef} we just add the terms
$$\normom{\div\eps(E-\tilde{E})}=\normom{\div\eps\tilde{E}}=\normom{\div\eps(1-\pi)\tilde{E}}$$
and the sum of
$$\left|\scpom{\eps(E-\tilde{E})}{H_{n}}\right|
=\left|\scpom{\eps\tilde{E}}{H_{n}}\right|
=\left|\scpom{\eps(1-\pi)\tilde{E}}{H_{n}}\right|.$$
Of course, the terms in the first equalities make sense for $\eps\tilde{E}\in\Hdom{}$ only.
\end{rem}

\begin{theo}
\mylabel{esttwo}
Let $E, \tilde E$ be as in Lemma \ref{estimatelem}
and additionally $\tilde{E}\in\Hcdepsom{}$. Then 
\begin{align}
\begin{split}
\tnormom{E-\tilde{E}}
&\leq M_{+}(\tilde{E};Y)
:=\frac{c_p}{\sqrt{c_{\mu}}}\normom{F-\curl Y}
+\normmuom{Y-\mu^{-1}\curl\tilde{E}}\\
&\qquad+2c_\gamma\norm{G-\ttrg\tilde{E}}_{\Htmotcg}
+\normom{\div\eps\tilde{E}}
+\sum_{n=1}^{d_{D}}\left|\scpom{\eps\tilde{E}}{H_{n}}\right|
\end{split}\mylabel{estprotwo}
\end{align}
holds for any $Y\in\Hcom{}$. 
If $E-\tilde{E}$ even belongs to $\Hccdepsom{}$, i.e. if the
approximation $\tilde{E}$ satisfies the boundary condition exactly, then
$\tnormom{\,\cdot\,}$ is a norm for $E-\tilde{E}$ and we have for all $Y\in\Hcom{}$
\begin{align}
\begin{split}
\tnormom{E-\tilde{E}}
&\leq M_{+}(\tilde{E};Y)
=\frac{c_p}{\sqrt{c_{\mu}}}\normom{F-\curl Y}
+\normmuom{Y-\mu^{-1}\curl\tilde{E}}\\
&\qquad+\normom{\div\eps\tilde{E}}
+\sum_{n=1}^{d_{D}}\left|\scpom{\eps\tilde{E}}{H_{n}}\right|.
\end{split}\mylabel{estprothree}
\end{align}
\end{theo}

\begin{rem}
\mylabel{estonerem}
If $\tilde{E}$ satisfies the prescribed boundary condition
and $\eps\tilde{E}$ is solenoidal and perpendicular to Dirichlet fields,
then \eqref{estproone} or \eqref{estprotwo}, \eqref{estprothree} imply
for all $Y\in\Hcom{}$
\begin{align}
\begin{split}
\tnormom{E-\tilde{E}}&=\normmumoom{\curl(E-\tilde{E})}\\
&\leq M_{+}(\tilde{E};Y)
=\frac{c_p}{\sqrt{c_{\mu}}}\normom{F-\curl Y}
+\normmuom{Y-\mu^{-1}\curl\tilde{E}}
\end{split}\mylabel{estproonem} 
\end{align}
and the left hand side is a norm for $E-\tilde{E}$.
The estimates \eqref{estproone}-\eqref{estproonem} show that deviations from
exact solutions contain weighted residuals
of basic relations with weights given by constants in the corresponding
embedding inequalities.
These are typical features of the so-called functional a posteriori error estimates.
\end{rem}

\begin{rem}
\mylabel{estthreeremthree}
We see that $M_{+}(\tilde{E};Y)=0$, if and only if
$\tilde{E}:=E$ and $Y:=\mu^{-1}\curl E$
(by Lemma \ref{weakcurllem} we then have $Y\in\Hcom{}$).
Moreover, we note that \eqref{estprothree} and \eqref{estproonem} are sharp,
which easily can be seen by setting 
$Y:=\mu^{-1}\curl E\in\Hcom{}$.
In other words, if $\tilde{E}\in\Hcdepsom{}$ 
satisfies the boundary condition exactly then
$$\tnormom{E-\tilde{E}}=\inf_{Y\in\Hcom{}}M_{+}(\tilde{E};Y).$$
\end{rem}

\begin{rem}
\mylabel{estthreeremfour}
In Theorem \ref{estone} and Theorem \ref{esttwo}
we can replace the boundary term on the right hand side
by $2\normmumoom{\curl T}$ or $2\normmumoom{\curl\textg(G-\ttrg\tilde{E})}$
using Lemma \ref{estimatelem} and Corollary \ref{estimatecor}.
Especially for numerical applications the first choice is recommendable.
Hence, we may assume that $G$ is always given by a tangential trace of
some vector field $\check{G}\in\Hcom{}$, i.e. $\ttrg\check{G}=G$.
Then $T:=\check{G}-\tilde{E}$ and we do not have to know the constant $c_\gamma$. 
\end{rem}

\begin{rem}
\mylabel{estthreeremdirichlet}
If the domain is 'simple' in terms of a vanishing second Betti number,
i.e. there are no 'handles', then there exist no Dirichtlet fields.
Thus, for instance, in Theorem \ref{esttwo}
the last summand in the respective estimates does not occur.
\end{rem}

\section{Lower bounds for the error}\mylabel{lowboundsec}

Now, we proceed to derive computable lower bounds of the error.
First, we present the following subsidiary result:

\begin{lem}
\mylabel{weakcurllem}
If $E$ satisfies \eqref{varformu} then $\mu^{-1}\curl E\in\Hcom{}$
and $\curl\mu^{-1}\curl E=F$.
\end{lem}

\begin{proof}
We need to show that
\begin{align}
\scpom{\mu^{-1}\curl E}{\curl\Phi}&=\scpom{F}{\Phi}\quad\forall\,\Phi\in\Cizom.\mylabel{varformuci}
\end{align}
Using $\pi$ from Remark \ref{helmholtzrem}, we obtain 
$W=\pi\Phi\in\Hccdzbotom{}$ provided that $\Phi\in\Cizom$.
Thus, by \eqref{varformu} and the fact that $\curl(1-\pi)\Phi=0$, we get
\begin{align}
\scpom{\mu^{-1}\curl E}{\curl\Phi}=\scpom{\mu^{-1}\curl E}{\curl\pi\Phi}
&=\scpom{F}{\pi\Phi}\quad\forall\,\Phi\in\Cizom.\mylabel{varformupici}
\end{align}
Since $F\in\Hdzbotom{}=\ol{\curl\Hcom{}}$, we get (by approximation)
$\scpom{F}{\pi\Phi}=\scpom{F}{\Phi}$ and \eqref{varformuci} follows.
To be more precise, we select $F_{n}\in\Hcom{}$, for which
$(\curl F_{n})_{n\in\nz}$ converges in $\Hom{}$ to $F$, using
$\pi\Phi\in\Hccom{}$ and $\curl(1-\pi)\Phi=0$. Then
$$\scpom{\curl F_{n}}{\pi\Phi}
=\scpom{F_{n}}{\curl\pi\Phi}
=\scpom{F_{n}}{\curl\Phi}
=\scpom{\curl F_{n}}{\Phi}
\quad\forall\,\Phi\in\Cizom.$$
\end{proof}

Lemma \ref{weakcurllem} implies

\begin{rem}
\mylabel{weakcurlrem}
Let $E\in\Hcom{}$ and some $F$ be given. Then the following three assertions are equivalent:
\begin{itemize}
\item[\bf(i)] $\mu^{-1}\curl E\in\Hcom{}$ and $\curl\mu^{-1}\curl E=F$.
\item[\bf(ii)] $F\in\Hom{}$ and
$$\scpom{\mu^{-1}\curl E}{\curl\Phi}=\scpom{F}{\Phi}\quad\forall\,\Phi\in\Hccom{}.$$
\item[\bf(iii)] $F\in\Hdzbotom{}$ and 
$$\scpom{\mu^{-1}\curl E}{\curl\Phi}=\scpom{F}{\Phi}\quad\forall\,\Phi\in\Hccdzbotom{}.$$
\end{itemize}
\end{rem}

\begin{theo}
\mylabel{lowerboundtheo}
Let $\tilde{E}\in\Hcom{}$ be an approximation. Then
$$\normmumoom{\curl(E-\tilde{E})}^2\geq\sup_{W}M_{-}(\tilde{E};W),$$
where 
$$M_{-}(\tilde{E};W):=2\scpom{F}{W}-\scpom{\mu^{-1}\curl(2\tilde{E}+W)}{\curl W}$$
and the supremum is taken over $\Hccom{}$.
This estimate is sharp if $E-\tilde{E}$ belongs to the latter space,
i.e. if the approximation $\tilde{E}$ satisfies the boundary condition exactly.
\end{theo}

\begin{proof}
We start with the obvious identity
\begin{align*}
\normmumoom{\curl(E-\tilde{E})}^2
&=\sup_{Y\in\Hom{}}\left(2\scpom{\mu^{-1}\curl(E-\tilde{E})}{Y}-\normmumoom{Y}^2\right).
\end{align*}
Thus, for all $W\in\Hcom{}$ we obtain the estimate
\begin{align*}
\normmumoom{\curl(E-\tilde{E})}^2
&\geq2\scpom{\mu^{-1}\curl(E-\tilde{E})}{\curl W}-\normmumoom{\curl W}^2\\
&=2\scpom{\mu^{-1}\curl E}{\curl W}-\scpom{\mu^{-1}\curl(2\tilde{E}+W)}{\curl W}.
\end{align*}
Clearly, this estimate is sharp because we can always put $W=E-\tilde{E}$.
However, to exclude the unknown exact solution $E$ from the right hand side
we need $W\in\Hccdzbotom{}$. Then, by \eqref{varformu}
\begin{align}
\scpom{\mu^{-1}\curl E}{\curl W}=\scpom{F}{W}\mylabel{estbelow}
\end{align}
and by Lemma \ref{weakcurllem} \eqref{estbelow} even holds for all $W\in\Hccom{}$.
Thus, for all $W\in\Hccom{}$
\begin{align}
\normmumoom{\curl(E-\tilde{E})}^2
&\geq M_{-}(\tilde{E};W).\mylabel{lowbdmm}
\end{align}
Obviously, this lower bound is sharp if
we can set $W=E-\tilde{E}\in\Hccom{}$.
\end{proof}

The following result is trivial:

\begin{cor}
\mylabel{lowerboundcor}
Let $\tilde{E}\in\Hcdepsom{}$ be an approximation. Then
$$\tnormom{E-\tilde{E}}^2
\geq\sup_{W\in\Hccom{}}M_{-}(\tilde{E};W)
+\normom{\div\eps\tilde{E}}^2
+\sum_{n=1}^{d_{D}}\scpom{\eps\tilde{E}}{H_{n}}^2.$$
\end{cor}

\begin{rem}
\mylabel{lowerboundtheorem}
In general, the lower bound is not sharp because
if $E-\tilde{E}\notin\Hccom{}$ 
then we can not put $W=E-\tilde{E}$ anymore.
In fact, with $\mu^{-1}\curl E\in\Hcom{}$ and $\curl\mu^{-1}\curl E=F$
by Lemma \ref{weakcurllem} we get for all $W\in\Hcom{}$
\begin{align}
\scpom{\mu^{-1}\curl E}{\curl W}
=\scpom{F}{W}+\scpg{\tttrg\mu^{-1}\curl E}{\ttrg W}.\mylabel{intpart}
\end{align}
Here, we introduced a second tangential trace $\tttrg$,
called the normal trace in terms of differential forms,
mapping again $\Hcom{}$ to special tangential vector fields on the boundary as well, i.e.
$$\Htmotdg:=\Set{\psi\in\Hgen{-1/2}{t}{}(\gamma)}{\divs\psi\in\Hgen{-1/2}{}{}(\gamma)},$$
where $\divs$ denotes the surface divergence.
In \cite{buffaciarlet,buffacostabelsheen} and, more general,
in \cite{wecklip}, see also \cite[Lemma 3.7, $q=1$]{paulyrossi},
it has been pointed out that even for Lipschitz domains
the integration by parts formula \eqref{intpart} remains valid
in some sophisticated sense.
For smooth vector fields we have $\tttrg=-n\times(n\times\,\cdot\,)|_{\gamma}$.
Hence, we obtain the estimate
$$\normmumoom{\curl(E-\tilde{E})}^2
\geq M_{-}(\tilde{E};W)
+2\scpg{\tttrg\mu^{-1}\curl E}{\ttrg W}$$
for all $W\in\Hcom{}$, which is sharp and coincides with \eqref{lowbdmm}
if $W\in\Hccom{}$.
But the unknown exact solution $E$
still appears on the right-hand side,
i.e. the second tangential trace of $\mu^{-1}\curl E$ on $\gamma$.
Furthermore, if the term $\scpg{\tttrg\mu^{-1}\curl E}{\trg W}$
is positive then \eqref{estbelow} can not be sharp.
\end{rem}

\begin{acknow}
The authors express their gratitude
to the Department of Mathematical Information Technology
of the University of Jyv\"askyl\"a (Finland)
for financial support.
\end{acknow}

\addressesdirksergey

\end{document}

%% file: head_paper.tex
\usepackage{a4,exscale,ifthen,amsfonts,amssymb,amsmath,amscd,graphicx}
\usepackage[english]{babel}
\usepackage[mathscr]{eucal}

\author{{\sf\pauthor}}
\numberwithin{equation}{section}

\newenvironment{acknow}{{\vspace*{1cm}\noindent\bf Acknowledgements }}{}
\newcommand{\bewboxw}{\mbox{}\hfill $\square$ \\}

\newenvironment{proof}{{\noindent\bf Proof }}{\bewboxw}

\newcommand{\keywords}[1]{{\noindent\bf Key Words }#1}
\newcommand{\amsclass}[1]{{\noindent\bf AMS MSC-Classifications }#1}

\ifthenelse{\equal{\mylabelonoff}{on}}
{\newcommand{\mylabel}[1]{\label{#1}\fbox{{\rm #1}}}}{\newcommand{\mylabel}[1]{\label{#1}\makebox[0mm][]{}}}
\ifthenelse{\equal{\allowdisbrk}{yes}}
{\allowdisplaybreaks}{}

\newcommand{\paper}[7]{\bibitem{#1} #2, `#3', {\it #4}, #5, (#6), #7.}
\newcommand{\paperpub}[5]{\bibitem{#1} #2, `#3', {\it #4}, (#5), accepted for publication.}

\newcommand{\book}[6]{\bibitem{#1} #2, {\it #3}, #4, #5, (#6).}

\newcommand{\dissavail}[6]{\bibitem{#1} #2, `#3', {\sf Dissertation}, #4, (#5), available from {\tt #6}.}

\newcommand{\repjyu}[8]{\bibitem{#1} #2, `#3', {\it Reports of the Department of Mathematical Information Technology}, 
University of Jyv\"askyl\"a, Series #4. Scientific Computing, No. #4. #5/#6, ISBN #7, ISSN #8.}

%% file: head_dirk_sergey.tex
\usepackage[mathscr]{eucal}

\newcommand{\ol}{\overline}

\newcommand{\nz}{\mathbb{N}}

\newcommand{\rz}{\mathbb{R}}

\newcommand{\rt}{\rz^3}

\DeclareMathOperator{\grad}{grad}
\DeclareMathOperator{\curl}{curl}

\DeclareMathOperator{\curls}{\curl_{s}}
\renewcommand{\div}{\operatorname{div}}
\DeclareMathOperator{\divs}{\div_{s}}

\newcommand{\eps}{\varepsilon}

\newcommand{\om}{\Omega}

\newcommand{\trg}{\tau_\gamma}

\newcommand{\ttrg}{\tau_{t,\gamma}}

\newcommand{\tttrg}{\tilde{\tau}_{t,\gamma}}

\newcommand{\textg}{\check{\tau}_{t,\gamma}}

\newcommand{\Eg}{E_{\gamma}}

\def\Set#1#2{\left\{#1\,\mid\,#2\right\}}

\DeclareMathOperator{\Lebesgue}{\mathsf{L}}
\newcommand{\Lgen}[2]{\Lebesgue^{#1}_{#2}}

\def\Lt{\Lgen{2}{}}

\def\Ltom{\Lt(\om)}

\DeclareMathOperator{\Sobolev}{\mathsf{H}}
\newcommand{\Hgen}[3]{\overset{#3}{\Sobolev}{}^{#1}_{#2}}

\newcommand{\Htmotcg}{\Hgen{-1/2}{t}{}(\curls,\gamma)}
\newcommand{\Htmotdg}{\Hgen{-1/2}{t}{}(\divs,\gamma)}

\def\Hoz{\Hgen{1}{}{\circ}}

\def\Hozom{\Hoz(\om)}

\DeclareMathOperator{\Cont}{\mathsf{C}}
\newcommand{\Cgen}[2]{\overset{#2}{\Cont}{}^{#1}}

\def\Ciz{\Cgen{\infty}{\circ}}

\def\Cizom{\Ciz(\om)}

\newcommand{\harmdieps}{\mathcal{H}_{D,\eps}(\om)}

\newcommand{\harmnemu}{\mathcal{H}_{N,\mu}(\om)}

\newcommand{\Hggenom}[2]{\Sobolev_{#1}(\grad^{#2},\om)}
\newcommand{\Hgcom}[1]{\Hggenom{#1}{\circ}}
\newcommand{\Hcgenom}[3]{\Sobolev_{#1}(\curl^{#2}_{#3},\om)}
\newcommand{\Hcom}[1]{\Hcgenom{#1}{}{}}
\newcommand{\Hccom}[1]{\Hcgenom{#1}{\circ}{}}
\newcommand{\Hczom}[1]{\Hcgenom{#1}{}{0}}
\newcommand{\Hcczom}[1]{\Hcgenom{#1}{\circ}{0}}
\newcommand{\Hdgenom}[3]{\Sobolev_{#1}(\div^{#2}_{#3},\om)}
\newcommand{\Hdom}[1]{\Hdgenom{#1}{}{}}

\newcommand{\Hdzom}[1]{\Hdgenom{#1}{}{0}}
\newcommand{\Hdczom}[1]{\Hdgenom{#1}{\circ}{0}}
\newcommand{\Hcdzbotom}[1]{\Sobolev_{#1}(\curl,\div_{0}\eps,\bot_{\eps},\om)}
\newcommand{\Hccdzbotom}[1]{\Sobolev_{#1}(\curl^{\circ},\div_{0}\eps,\bot_{\eps},\om)}
\newcommand{\Hdzbotom}[1]{\Sobolev_{#1}(\div_{0},\bot,\om)}
\newcommand{\Hdzbotepsom}[1]{\Sobolev_{#1}(\div_{0}\eps,\bot_{\eps},\om)}

\newcommand{\Hcdepsom}[1]{\Sobolev_{#1}(\curl,\div\eps,\om)}
\newcommand{\Hccdepsom}[1]{\Sobolev_{#1}(\curl^{\circ},\div\eps,\om)}

\newcommand{\Hom}[1]{\Hgen{}{#1}{}(\om)}

\newcommand{\normdst}{\hspace{-0.4ex}}
\newcommand{\scp}[2]{\left\langle#1,#2\right\rangle}
\newcommand{\scpom}[2]{\scp{#1}{#2}_{\om}}

\newcommand{\scpg}[2]{\scp{#1}{#2}_{\gamma}}

\newcommand{\norm}[1]{\left|\normdst\left|#1\right|\normdst\right|}
\newcommand{\normom}[1]{\norm{#1}_{\om}}

\newcommand{\normmuom}[1]{\norm{#1}_{\mu,\om}}
\newcommand{\normmumoom}[1]{\norm{#1}_{\mu^{-1},\om}}

\newcommand{\tnorm}[1]{\left|\normdst\left|\normdst\left|#1\right|\normdst\right|\normdst\right|}
\newcommand{\tnormom}[1]{\tnorm{#1}_{\om}}

\newtheorem{lem}{Lemma}%[section]

\newtheorem{theo}[lem]{Theorem}
\newtheorem{cor}[lem]{Corollary}
\newtheorem{rem}[lem]{Remark}

\newcommand{\addressesdirksergey}{
\vspace*{2cm}
\begin{center}
\begin{tabular}{ll}
{\sf Dirk Pauly} & {\sf Sergey Repin} \\
\\
Universit\"at Duisburg-Essen & V.A. Steklov Mathematical Institute \\
Campus Essen & St. Petersburg Branch \\
Fakult\"at f\"ur Mathematik & \\
Universit\"atsstr. 2 & Fontanka 27 \\
45117 Essen & 191011 St. Petersburg \\
Germany & Russia \\
e-mail: dirk.pauly@uni-due.de & e-mail: repin@pdmi.ras.ru \\
\\
& and\\
\\
& University of Jyv\"askyl\"a \\
& Faculty of Information Technology \\
& Department of \\
& Mathematical Information Technology \\
& P.O. Box 35 (Agora) \\
& FI-40014 Jyv\"askyl\"a \\
& Finland \\
& e-mail: serepin@jyu.fi
\end{tabular}
\end{center}}